\documentclass[11pt, french,english]{article}
\usepackage{amsmath,bm,epsfig,amssymb,mathrsfs,amsthm,amstext,esint,babel,amsfonts,mathtools,epstopdf,ifthen}
\usepackage{geometry} \usepackage[T1]{fontenc} \usepackage[latin9]{inputenc} \usepackage{graphicx,color,soul,dsfont,enumerate}
\usepackage{subfloat,subcaption}
\usepackage{enumerate}
\usepackage{pifont}
\usepackage{bbm}
\usepackage{upgreek,bm}
\usepackage{arydshln}
\newcommand{\cmark}{\ding{51}}%
\newcommand{\xmark}{\ding{55}}%

\newtheorem{proposition}{\textbf{Proposition}}
\newtheorem{lemma}[proposition]{\textbf{Lemma}}
\newtheorem{corollary}[proposition]{\textbf{Corollary}}
\newtheorem{theorem}[proposition]{\textbf{Theorem}}
\newtheorem{definition}[proposition]{\textbf{Definition}}

\def\R{{\mathcal{R}}}

\def\Z{ \mathbb{Z}}

\def\R{ \mathbb{R}}

\def\tvarphi{\tilde{\varphi}}
\def\tupvarphi{\tilde{\varphi}}
\def\jj{\mathrm{j}}
\def\ee{\mathrm{e}}
\def\R{\mathbb{R}}
\def\dint{\mathrm{d}}

\title{Support and Approximation Properties of Hermite Splines}
\author{Julien Fageot\footnote{Biomedical Imaging Group, \'Ecole Polytechnique F\'ed\'erale de Lausanne (EPFL), Station 17, 1015 Lausanne, Switzerland.}, Shayan Aziznejad\footnotemark[1], Michael Unser\footnotemark[1], Virginie Uhlmann\footnotemark[1] \footnote{European Bioinformatics Institute (EMBL-EBI), Wellcome Genome Campus, Cambridge CB10 1SD, UK.}  \footnote{Corresponding author. \newline%
Contact: julien.fageot@epfl.ch; shayan.aziznejad@epfl.ch; michael.unser@epfl.ch; uhlmann@ebi.ac.uk.}  %
}

\begin{document}
\maketitle

\textbf{Abstract:} In this paper, we 
formally investigate two mathematical aspects of Hermite splines which translate to features that are relevant to their practical applications.
We first demonstrate that Hermite splines are maximally localized in the sense that their support sizes are minimal among pairs of functions with identical reproduction properties.
Then, we precisely quantify the approximation power of Hermite splines for reconstructing functions and their derivatives, and show that they are asymptotically identical to cubic B-splines for these tasks.
Hermite splines therefore combine optimal localization and excellent approximation power, while retaining interpolation properties and closed-form expression, in contrast to existing similar approaches.
These findings shed a new light on the convenience of Hermite splines for use in computer graphics and geometrical design.

\section{Introduction}
In his seminal $1973$ monograph on cardinal interpolation and spline functions~\cite{Schoenberg1973book}, I.J. Schoenberg explains and characterizes B-spline interpolation, which still inspire researchers and yield exciting applications nowadays. In the same work, he also sets the basis of Hermite interpolation~\cite{Lipow1973,Schoenberg1973}. In the classical B-spline framework, a continuous-domain function is constructed from a discrete sequence of samples. By contrast, the Hermite interpolation problem involves two sequences of discrete samples, which impose constraints not only on the resulting interpolated function but also on its derivatives up to a given order. 

Curves in the plane or surfaces in the volume can be constructed from one-dimensional interpolation schemes $\mathbb{R}\rightarrow\mathbb{R}$ by interpolating along each spatial coordinate. The practical value of Hermite splines in this context is to offer tangential control on the interpolated curve. This can be easily understood through their link with B\'ezier curves~\cite{Farouki2012}. The latter lie at the heart of vector graphics and are popular tools for computer-aided geometrical design and modelling \cite{Prautzsch2013, Farin2002, Bohm1984}. 
Hermite splines are also an interesting option for the design of multiwavelets, that is, wavelets with multiple generators~\cite{Dahmen2000,Warming2000} because of their small support.
In practice, Hermite splines thus provides a suitable solution to a number of problems, whether with respect to simplicity of construction, efficiency, or convenience. 
This hands-on intuition can be translated to formal properties of Hermite splines and mathematically characterized.
As examples, the joint interpolation properties of Hermite splines (see Section~\ref{sec:introHer}), which ensure that, at integer values, the interpolated function exactly matches the sequences of samples and derivatives samples that were used to build it; their smoothness properties~\cite{Uhlmann2016}, which guarantee low curvature of the interpolated curve under some mild conditions; and their statistical optimality (in terms of MMSE) for the reconstruction of second-order Brownian motion from direct and first derivative samples~\cite{Uhlmann2015}.
In that spirit, we investigate in this work the theoretical counterpart of two additional features that are observed to grant Hermite splines their practical usefulness. 

\subsection{Contributions}
Our contributions state the minimal support property of Hermite splines and investigate their approximation power. In the following, we describe the practical observations motivating them, the results themselves, and related works.

\paragraph{Minimal support property.} The short support of Hermite splines is an important feature making them attractive in practice. The size of the support relates to the local extent of modifications on the continuously-defined spline curve. 
We formally demonstrate that Hermite splines have the minimal support property among basis functions generating cubic and quadratic splines. 
When dealing with B-splines, there is a trade-off between the ability to reproduce smooth functions, which only depends on the B-spline order and thus maximum polynomial degree, and the possibility to allow for sharp transitions, as the regularity of the spline increases with its order. Typically, cubic splines can efficiently reproduce smooth functions, but lack the power to capture irregular transitions. On the other hand, quadratic splines have lesser approximation power, but are preferred when dealing with sharp transitions.
Hermite splines combine these two strengths in one scheme and are, in terms of support size, better than any two-functions scheme, including the one composed of the classical cubic and quadratic B-splines. In addition, we also show that one necessarily requires two generators to achieve this optimality. This result relates to similar ones involving a single generator~\cite{DelgadoGonzalo2012,Blu2001}.

\paragraph{Rate of decay of the approximation error.} 
Hermite splines can provide faithful approximation reasonably fast as the number of parameters increases. 
 This feature relates to the rate of decay of the approximation error. Numerous works approach these questions by restricting themselves to a specific interpolation framework, such as~\cite{Hall1968,Birkhoff1967} for the specific case of Hermite approximation. They provide precise estimations of the optimal bound on the approximation error relying on $L_\infty$ norms, but do not allow comparison with other schemes. In contrast, approaches have been developed for single~\cite{Blu1999a} and multiple generators~\cite{Blu1999b} to provide a unifying setting for comparison. Their rigorous mathematical framework offer generalized measures that can be applied to a wide variety of basis functions for estimating approximation constants. Hermite interpolation however violates some of the core assumptions of the analysis, preventing it from being studied with these tools.
Taking strong inspiration from this previous work, we provide a novel study of Hermite approximation with a analogous analysis strategy, relieving the boundedness assumptions and allowing to consider additional spline approximation schemes.
We precisely quantify the rate of decay of the approximation error of the Hermite scheme and quantitatively estimate the corresponding approximation constants. Hermite splines have excellent approximation properties when it comes to reconstructing a function and its first derivative, actually being close to optimal. 
The investigation of approximation error on the derivative relates to~\cite{condat2011reconstruction,condat2011quantitative}, although following a completely different line: while they focused on the reconstruction of the derivative from signal samples, the Hermite scheme grants direct access on the function and derivative samples, allowing to reconstruct the derivative in a multi-function setting.
\newline\newline
Pioneering works cover the study of spline schemes to an impressive degree of generality~\cite{de1994approximation,deboor1994structure,de1998approximation,holtz2005approximation}. They include results on minimum support and approximation errors in a large variety of cases. Regarding minimum support, these previous results are however restricted to single generator schemes. To the best of our knowledge, we are not aware of previous work considering multi-generators and therefore covering the Hermite case. 
These studies also do not focus on providing a way to quantify the approximation error constants, thus not enabling comparison between schemes of the same order. As mentioned beforehand, frameworks have been proposed to address this, but the Hermite scheme falls outside of their core hypotheses, therefore requiring a formal adaptation.

\subsection{Hermite Splines}\label{sec:introHer}
Schoenberg defines the cardinal cubic Hermite interpolation problem as follows~\cite{Lipow1973,Schoenberg1973}. Knowing the discrete sequences of numbers $c[k]$ and $d[k]$, $k\in\mathbb{Z}$, we look for a continuously-defined function $f_\mathrm{Her}(t)$, $t\in\mathbb{R}$, satisfying $f_\mathrm{Her}(k)=c[k]$, $f'_\mathrm{Her}(k)=d[k]$ for all $k\in\mathbb{Z}$, such that $f_\mathrm{Her}$ is piecewise polynomial of degree at most $3$ and once differentiable with continuous derivative at the integers.
The existence and uniqueness of the solution is guaranteed~\cite[Theorem $1$]{Lipow1973} for any sequences $c = (c[k])$ and $d= (d[k])$ bounded by a polynomial, but we shall restrict to sequences in $\ell_2(\mathbb{Z})$ thereafter. 
In \cite{Schoenberg1973}, it is shown that the Hermite spline $f_{\mathrm{Her}}$ associated to the sequences $c$ and $d$ can be expressed as 
\begin{align} \label{eq:fHercd}
	f_{\mathrm{Her}}(t) = \sum_{k \in \Z}  c[k]\phi_1(\cdot-k) + d[k]\phi_2(\cdot-k)
\end{align}
where the functions $\phi_1$ and $\phi_2$ are given by
\begin{align} 
\phi_1(t)&=(2|t|+1)(|t|-1)^2 \mathbbm{1}_{0\leq |t|\leq 1}, \label{eq:phi1} \\
\phi_2(t)&= t(|t|-1)^2  \mathbbm{1}_{0\leq |t| \leq 1}. \label{eq:phi2}
\end{align}
In addition to their fairly simple analytical expression, the cubic Hermite splines have other important properties. First, they are of finite support in $[-1,1]$. Moreover, the generating functions $\phi_1$, $\phi_2$ and their derivatives $\phi'_1$, $\phi'_2$ satisfy the joint interpolation conditions 
\begin{align}\label{eq:interpcond}
\phi_1(k)=\delta[k], \quad\phi'_2(k)=\delta[k], \quad
\phi'_1(k)=0, \quad
\phi_2(k)=0,
\end{align}
for all $k \in \mathbb{Z}$, where $\delta[k]$ is the discrete unit impulse. The functions and their first derivative are depicted in Figure~\ref{fig:Hermite}, where the interpolation properties can easily be observed. The functions $\phi_1$ and $\phi_2$ are deeply intertwined
as $c[k] = f_{\mathrm{Her}}(k)$ and $d[k] = f'_{\mathrm{Her}}(k)$ in \eqref{eq:fHercd}. 
The cubic Hermite splines are differentiable with continuous derivatives at the integer knots points $t=k$. As a result, functions generated by cubic Hermite splines are $C^1$-continuous piecewise-cubic polynomials with knots at integer locations. 


\begin{figure}
\begin{center}
\begin{subfigure}[b]{0.45\textwidth}
                \includegraphics[width=\textwidth]{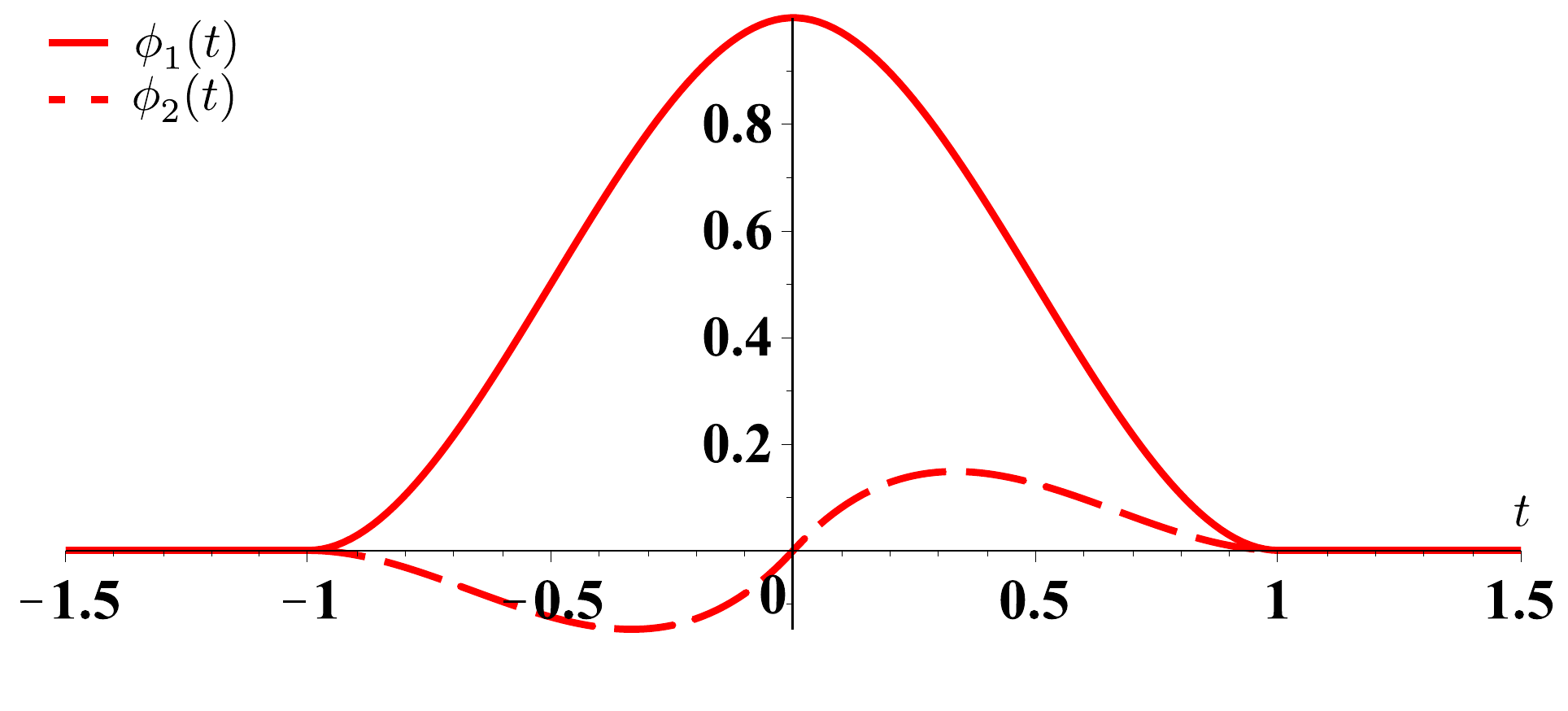}
                \label{fig:phi}
\end{subfigure}
\begin{subfigure}[b]{0.45\textwidth}
                \includegraphics[width=\textwidth]{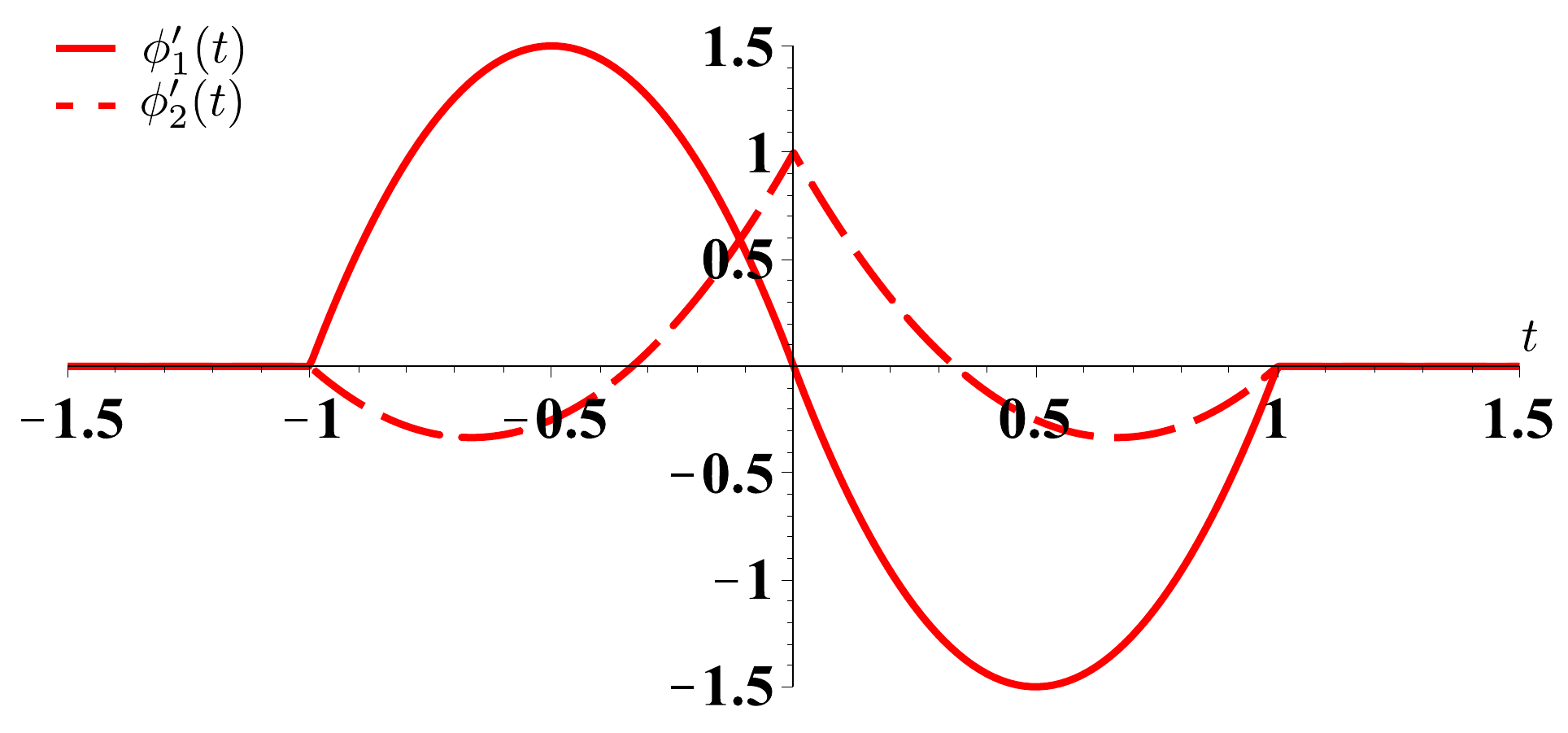}
                \label{fig:phip}
\end{subfigure}
\caption{\label{fig:Hermite}Cubic Hermite splines $\phi_1$ and $\phi_2$. The two functions and their derivatives are vanishing at the integers with the exception of $\phi_1(0)=1$ and $\phi_2'(0)=1$ (interpolation properties). They are supported in $[-1, 1]$.}
\end{center}
\end{figure}

\section{Minimum Support Properties of Hermite Splines} \label{sec:supporthermite}
	B-splines are known to be maximally localized, meaning that they are compactly supported with minimal support properties among functions with the same approximation properties~\cite{ron1990factorization,blu2001moms}. 
	Hermite splines possess a similar fundamental minimal support property: they are of minimal support among the pair of functions being able to generate both quadratic and cubic splines (Theorem \ref{theo:minimalsupport}). In addition, a single generating function is not sufficient for this purpose (Proposition \ref{prop:lesdeuxfontlapaires}).
	This demonstrates that Hermite splines are maximally localized for the purpose of representing piecewise-cubic functions that are continuously differentiable, as exploited for instance in image processing~\cite{Uhlmann2016}.

	\subsection{Shift invariant spaces and support properties}\label{sec:shiftinvariant}	
	
Consider a set of $N \geq 1$  functions $\bm{\varphi}=(\varphi_1 , \ldots , \varphi_N) \in (L_2(\R))^N$ and define the space of functions
		\begin{align}\label{eq:outspace}
		V(\bm{\varphi})&=\left\{ \sum_{i=1}^{N}\sum_{k\in\mathbb{Z}} c_i[k]\varphi_i( \cdot- k) : \mathbf{c} = (c_1, \ldots , c_N) \in (\ell_2(\mathbb{Z}))^{N} \right\} .
		\end{align}
		We say that $\varphi_i$ are \emph{basis functions}  of the set $V(\bm{\varphi})$. 
		The space $V(\bm{\varphi})$ is \emph{shift-invariant} in the sense that, by definition, $f(\cdot - k)$ is in $V(\bm{\varphi})$ for $f \in V(\bm{\varphi})$ and $k \in \Z$~\cite{de1994approximation,deboor1994structure,de1998approximation,jetter2001shift}.
	
		We consider shift-invariant spaces generated by a single ($N=1$) or two functions ($N=2$), corresponding to B-splines and Hermite splines, respectively.
		We only consider spaces \eqref{eq:outspace} for which 
		the family $\left(\varphi_{i} (\cdot - k) \right)_{i=1\cdots N, k \in \Z}$ is a Riesz basis; that is,  
\begin{align} \label{eq:Rieszcondition}
A\sum_{i=1}^{N}\sum_{k\in\mathbb{Z}} c_i[k]^2 \leq \left\|\sum_{i=1}^{N}\sum_{k\in\mathbb{Z}} c_i[k]\varphi_i( \cdot - k)\right\|^2_{L_2} \leq B \sum_{i=1}^{N}\sum_{k\in\mathbb{Z}} c_i[k]^2
\end{align} 
for any $c_i\in\ell_2(\mathbb{Z})$ with $i=1,\ldots,N$ and for some constants  $0<A\leq B<\infty$.
		This ensures that any $f \in V(\bm{\varphi})$ has a unique and stable representation.
		
		
		\paragraph{Support properties of B-splines.}
		Under the Riesz condition, a natural question is the ability of the basis functions $\bm{\varphi}$ to exactly reproduce classes of functions. 
		The possibility to perfectly reproduce polynomials is of crucial importance.
		The constant function $1$ can be reproduced if and only if the basis functions satisfy the partition of unity, which is the minimal requirement for a practical approximation scheme~\cite{de1985partitions}.
		
		The polynomial B-spline of order $L \geq 1$ is classically known to be able to reproduce polynomials up to degree $(L-1)$; that is, $t^{\ell}$ for every $\ell = 0, \ldots , (L-1)$. Following \cite{Unser1999}, we denote it by $\beta^{(L-1)}$.
		In particular, the cubic B-spline (of order $L=4$), can perfectly reproduce any polynomial of degree at most $3$. 
	The ability of the basis functions $\bm{\varphi}$ to perfectly reproduce polynomials is intimately linked to their approximation power, as will be developed in Section~\ref{sec:approximation}. B-splines are actually the most localized functions satisfying this property, as formalised in Proposition \ref{prop:BsplineMS}.
	
		\begin{proposition} \label{prop:BsplineMS}
		Let $\varphi \in L_2(\R)$ be a function such that $(\varphi(\cdot - k))_{k \in \Z}$ is a Riesz basis and can reproduce polynomials up to degree $(L-1) \geq 0$. 
		Then the support of $\varphi$ is at least of size $L$. 
		
		In particular, the B-spline of order $L$, whose support is of size $L$, is optimal in terms of support localization among basis functions that are able to reproduce polynomials of degree at most $(L-1)$. 
	\end{proposition}
	
		This result is classical in approximation theory: Schoenberg showed that B-splines effectively have the adequate approximation order~\cite{Schoenberg1973book}. A complete characterization of the functions of minimal support with a given approximation order can be found in~\cite[Theorem 1]{Blu2001}.
			To the best of our knowledge, very little is known about the localization of basis functions when $N>1$, which is what we propose to investigate.

	\subsection{Minimal Support Properties for Two Basis Functions}

	Hermite splines $\phi_1$ and $\phi_2$, given by \eqref{eq:phi1} and \eqref{eq:phi2}, are able to reproduce both $\beta^2$ and $\beta^3$, the quadratic and cubic B-splines of order $3$ and $4$, respectively~\cite{Uhlmann2016}.
	In particular, this means that $V(\phi_1,\phi_2)$ contains polynomials of degree at most $3$.
	Many other pairs of basis functions, starting with $\beta^2$ and $\beta^3$ themselves, can also reproduce quadratic and cubic splines.
	In line with Proposition \ref{prop:BsplineMS}, investigating the support localization of a pair of functions having the same reproduction properties as the Hermite splines follows naturally. This boils down to studying basis functions for which $\beta^2,\beta^3 \in V(\varphi_1,\varphi_2)$.
	
	We characterize the support size of such pairs of functions in Theorem \ref{theo:minimalsupport}. 
	This result then allows us to deduce the minimum support property of Hermite splines in Corollary \ref{coro:hermite}.
			
		\begin{theorem} \label{theo:minimalsupport}
			Let $\varphi_1, \varphi_2 \in L_2(\R^d)$ be two basis functions.
			We assume that 
			\begin{align}
				\beta^2(t) &= \sum_{k \in \Z} a_k \varphi_1 (t - k)  + b_k \varphi_2 (t - k),  \label{eq:beta2}\\
				\beta^3(t) &= \sum_{k \in \Z} c_k \varphi_1 (t - k)  + d_k \varphi_2 (t - k), \label{eq:beta3}
			\end{align}			
			with reproduction sequences $a,b,c,d$ satisfying
			\begin{align} \label{eq:conditionsequences}
				\sum_{k \in \Z} k^3 ( |a_k| + |b_k| + |c_k| + |d_k|) < \infty.
			\end{align}
			In particular, the quadratic and cubic B-splines $\beta^2,\beta^3$ are in $V(\varphi_1,\varphi_2)$.
			Then, we have that
			\begin{align}\label{eq:sizesupports}
			|\mathrm{Supp} \varphi_1 |
			+
				|\mathrm{Supp} \varphi_2 | \geq 4
			\end{align}
		\end{theorem}

\begin{proof}
		First of all, one can restrict to compactly supported basis functions $\varphi_1$ and $\varphi_2$ (otherwise, $|\mathrm{Supp} \varphi_1 | +	|\mathrm{Supp} \varphi_2 | = \infty$). 
		If one of the basis function, for instance $\varphi_2$, is identically zero, then the cubic spline $\beta^3 \in V(\varphi_1)$. 
		This means in particular that the basis function $\varphi_1$ reproduces polynomials up to degree 3, and its support is therefore at least of size four~\cite[Theorem 1]{Blu2001}.
		Hence, we again have that $|\mathrm{Supp} \varphi_1 | +	|\mathrm{Supp} \varphi_2 | = |\mathrm{Supp} \varphi_1 | \geq 4$. 
		We now assume that $\varphi_1$ and $\varphi_2$ are not identically $0$. \\
		
		\textit{Step 1.} We show that the  extreme points of the supports of $\varphi_1$ and $\varphi_2$ are integers.
		For $x=a,b,c,d$, we set $X(\omega) = \sum_{k\in \Z} x_k \mathrm{e}^{-\mathrm{j} \omega k}$, the $2\pi$-periodic Fourier transform of $x$. Condition \eqref{eq:conditionsequences} ensures that $X$ has a continuous third derivative. In Fourier domain, \eqref{eq:beta2} and \eqref{eq:beta3} become
		\begin{align}
			\widehat{\beta^2}(\omega) &= \frac{(1- \mathrm{e}^{- \mathrm{j}\omega})^3}{(\mathrm{j}\omega)^3} = A(\omega) \widehat{\varphi_1}(\omega) + B(\omega) \widehat{\varphi_2}(\omega),
			\label{eq:AB} \\
			\widehat{\beta^3}(\omega) &= \frac{(1- \mathrm{e}^{- \mathrm{j}\omega})^4}{(\mathrm{j}\omega)^4} = C(\omega) \widehat{\varphi_1}(\omega) + D(\omega) \widehat{\varphi_2}(\omega).
			\label{eq:CD}
		\end{align}
		We set $\mathrm{det}(\omega) = A(\omega) D(\omega) - B(\omega)C(\omega)$, which is itself a function with continuous third derivative.
		From \eqref{eq:AB} and \eqref{eq:CD}, we obtain
		\begin{align}
			\mathrm{det}(\omega) \widehat{\varphi}_1(\omega)& = D(\omega) \frac{(1- \mathrm{e}^{- \mathrm{j}\omega})^3}{(\mathrm{j}\omega)^3} - B(\omega)  \frac{(1- \mathrm{e}^{- \mathrm{j}\omega})^4}{(\mathrm{j}\omega)^4}, \label{eq:DB}\\
						\mathrm{det}(\omega) \widehat{\varphi}_2(\omega)& = -C (\omega) \frac{(1- \mathrm{e}^{- \mathrm{j}\omega})^3}{(\mathrm{j}\omega)^3} + A (\omega)  \frac{(1- \mathrm{e}^{- \mathrm{j}\omega})^4}{(\mathrm{j}\omega)^4}  . \label{eq:CA}
		\end{align}
		We show that $\mathrm{det}(\omega) \neq 0$ for $\omega \neq 0$ by contradiction. Let us fix $\omega_0 \in (0,2\pi)$ and assume that $\mathrm{det}(\omega_0) = 0$. We set $\alpha = \frac{1- \mathrm{e}^{-\mathrm{j} \omega_0}}{\mathrm{j}\omega_0}$ and $\beta = \frac{1- \mathrm{e}^{-\mathrm{j} \omega_0}}{\mathrm{j}\omega_0+2\pi}$. Then, $\alpha\neq 0$, $\beta\neq 0$, and $\alpha \neq \beta$. Moreover, \eqref{eq:DB} for $\omega = \omega_0$ and $(\omega_0+2\pi)$ implies that 
		\begin{align}
			\begin{pmatrix} 
\alpha^3 & -\alpha^4 \\
\beta^3 & -\beta^4 
\end{pmatrix}
 \begin{pmatrix} 
D(\omega_0) \\
B(\omega_0)
\end{pmatrix} 
= \begin{pmatrix} 
0\\
0
\end{pmatrix}.
		\end{align}
	The matrix being invertible (with determinant $\alpha^3\beta^3(\alpha-\beta) \neq 0$), we deduce that $D(\omega_0)= B(\omega_0)=0$. Similary, \eqref{eq:CA} with  $\omega = \omega_0$ and $(\omega_0+2\pi)$ implies that $A(\omega_0) = C(\omega_0) = 0$.  Injecting this in \eqref{eq:AB} with $\omega = \omega_0$, we deduce that $\widehat{\beta}_2(\omega_0) = \alpha^3 = 0$, which contradicts our initial assumption.
	
	We now study $\mathrm{det}(\omega)$ around the origin.
	Assume that $\mathrm{det}^{(p)}(0) = 0$ for $p = 0,1,2,3$.
	We will show that this is impossible.
	Under our assumption, Taylor's expansion gives that $\mathrm{det} (\omega) = o(\omega^3)$ around $0$. Hence, the function $\frac{\mathrm{det} (\omega) }{(1- \mathrm{e}^{- \mathrm{j} \omega})^3}$ can be continuously extended at $\omega = 0$, for which it is $0$. By periodicity, this is also valid for $\omega = 2 \pi$.
	In particular, $\frac{\mathrm{det} (\omega) }{(1- \mathrm{e}^{- \mathrm{j} \omega})^3} \widehat{\varphi}_1 (\omega)$ is continuous with limit $0$ at $\omega = 0$ and $2\pi$.
	From \eqref{eq:DB}, we moreover remark that
	\begin{align}
	\frac{\mathrm{det} (\omega) }{(1- \mathrm{e}^{- \mathrm{j} \omega})^3} \widehat{\varphi}_1 (\omega)
	=
	D(\omega) - B(\omega) \frac{1 - \mathrm{e}^{- \mathrm{j} \omega}}{\mathrm{j}\omega}.
	\end{align}
	Knowing the limit of the left term at $\omega = 0$ and $2\pi$, and using the $2\pi$-periodicity of $D(\omega)$, one deduces that $D(0) - B(0) = 0 = D(2\pi) - 0 = D(0)$. Therefore, $D(0) = B(0) = 0$. An identical reasoning based on \eqref{eq:CA} implies that $A(0) =C(0) = 0$. From \eqref{eq:AB} with $\omega = 0$, we obtain $\widehat{\beta}_3(0)=0$, which is false. 
	As a consequence, $\mathrm{det}^{(p)}(0) \neq 0$ for some $0\leq p \leq 3$, imposing that $\frac{(1- \mathrm{e}^{\mathrm{j} \omega})^3}{\mathrm{det}(\omega)}$ has a limit (possibly $0$) at $0$. Combined with the fact that $\mathrm{det}$ does not vanish over $(0,2\pi)$, we deduce that $F(\omega) =\frac{(1- \mathrm{e}^{\mathrm{j} \omega})^3}{\mathrm{det}(\omega)}$ is a continuous $2\pi$-periodic function.
	We can then rewrite \eqref{eq:DB}  as 
	\begin{align}
		(\mathrm{j}\omega)^4 \widehat{\varphi}_1(\omega) = (\mathrm{j}\omega) F(\omega) D(\omega) - (1-\mathrm{e}^{-\mathrm{j}\omega}) F(\omega) B(\omega). 
	\end{align}
	The functions $F(\omega)D(\omega)$ and $(1-\mathrm{e}^{-\mathrm{j}\omega}) F(\omega) B(\omega)$ are $2\pi$-periodic, hence their inverse Fourier transforms are sum of Dirac impulses located at the integers. It means in particular that we have, in time domain,
	\begin{align} \label{eq:varphi14} 
	\varphi_1^{(4)}(t) = \sum_{k\in\Z}  \left( y_k \delta(t-k) + z_k \delta'(t-k) \right).
	\end{align}
	Since $\varphi_1^{(4)}$ is compactly supported, like $\varphi_1$, only finitely many $y_k$ and $z_k$ are non-zero. Then, $\varphi_1$ is a compactly supported function whose fourth derivative has a support with integer extreme points (due to \eqref{eq:varphi14}), and therefore has a support with integer extreme points. The same reasoning also applies for $\varphi_2$, concluding this part of the proof. \\
	
	\textit{Step 2.} We know that the supports of $\varphi_1$ and $\varphi_2$ are of the form $[a,b]$ with $a<b$, $a,b\in\Z$.
	We assume that $|\mathrm{Supp} \varphi_1 |
			+
				|\mathrm{Supp} \varphi_2 | < 4$. Then, one of the two basis functions has a support of size one, for instance $\varphi_1$. We also assume without loss of generality that $\mathrm{Supp} \varphi_1 = [0,1]$, implying that only $y_0,y_1,z_0,z_1$ are possibly non-zero in \eqref{eq:varphi14}. Going back to the Fourier domain, one obtains
				\begin{align} \label{eq:beforeTaylor}
					(\mathrm{j}\omega)^4 \widehat{\varphi}_1(\omega) = y_0 + y_1 \mathrm{e}^{- \mathrm{j}\omega} +  \mathrm{j}\omega ( z_0 + z_1 \mathrm{e}^{- \mathrm{j}\omega}).
				\end{align}
				The function $\varphi_1$ is compactly supported. Its Fourier transform is hence infinitely smooth, and we can do the Taylor expansion of both sides in \eqref{eq:beforeTaylor}.
				In particular, we obtain the relations 
				\begin{align}
					y_0 + y_1 = z_0 + z_1 + y_1 = \frac{y_1}{2} + z_1 =  \frac{y_1}{6}  + \frac{z_1}{2} = 0.
				\end{align}
				This imposes that $y_0 = y_1 = z_0 = z_1 = 0$, which is absurd. Finally, it shows that $|\mathrm{Supp} \varphi_1 |+|\mathrm{Supp} \varphi_2 | \geq 4$, as expected. 				
\end{proof}
		
	Condition \eqref{eq:conditionsequences} plays an important role in our proof by imposing some regularity in Fourier domain.
	In practice, one even expects that compactly supported basis functions can generate the B-splines $\beta^2$ and $\beta^3$ with finitely many coefficients, in which case \eqref{eq:conditionsequences} is automatically satisfied.
	However, we believe that the condition \eqref{eq:conditionsequences} can be relaxed up to some extend.
	From Theorem \ref{theo:minimalsupport}, we easily deduce that Hermite splines have the minimum support property.
	
	\begin{corollary} \label{coro:hermite}
		The Hermite splines $(\phi_1,\phi_2)$ are of minimal support among the pair of functions that are able to reproduce both quadratic and cubic B-splines with reproduction sequences satisfying \eqref{eq:conditionsequences}. 
	\end{corollary}
	
	\begin{proof}
	From~\cite[Appendix A]{Uhlmann2016}, we know that Hermite splines can reproduce both quadratic and cubic B-splines, hence
	$\beta^2$ and $\beta^3$ are in $V(\phi_1,\phi_2)$ with compactly supported reproduction sequences obviously satisfying \eqref{eq:conditionsequences}. The supports of $\phi_1$ and $\phi_2$ are of size two, implying that $|\mathrm{Supp} \phi_1 | + |\mathrm{Supp} \phi_2 | = 4.$	Finally, the pair $(\phi_1,\phi_2)$ is maximally localized due to \eqref{eq:sizesupports}. 
	\end{proof}
	
	It is worth noting that the supports of the pair of Hermite splines jointly has the same size as the B-spline $\beta^3$. However, $\beta^3$ alone has less reproduction property. Being of class $C^2$, it is in particular unable to reproduce the quadratic spline $\beta^2$, which only has $C^1$ transitions at the integers. The simplest way of reproducing $\beta^2, \beta^3$ is to consider the basis pair $(\beta^2, \beta^3)$ itself, which is not maximally localized since the sum of the supports is $7$.
	An important additional remark is that two functions are needed to reproduce both cubic and quadratic spline, as formalised in Proposition \ref{prop:lesdeuxfontlapaires}.
	
	\begin{proposition} \label{prop:lesdeuxfontlapaires}
		There exists no single function $\varphi \in L_2(\R)$ that is able to reproduce $\beta^2$ and $\beta^3$ with summable reproduction sequences. 
	\end{proposition}
		
	\begin{proof} 
	By contradiction, let us assume that there exists $\varphi$ such that $\beta^2 = \sum_{k \in \Z} a_k \varphi( \cdot - k)$ and $\beta^3 = \sum_{k \in \Z} b_k \varphi( \cdot - k)$
	with $a,b \in \ell_1(\Z)$. Then, the Fourier transforms $A(\mathrm{e}^{\mathrm{j} \omega})$ and $B(\mathrm{e}^{\mathrm{j} \omega})$ are continous $2\pi$-periodic functions. In Fourier domain, we have that
	\begin{align}
		\left( \frac{1- \mathrm{e}^{-\mathrm{j}\omega}}{\mathrm{j}\omega}\right)^3 = A(\mathrm{e}^{\mathrm{j} \omega}) \widehat{\varphi}(\omega), \quad 
		\left( \frac{1- \mathrm{e}^{-\mathrm{j}\omega}}{\mathrm{j}\omega}\right)^4 = B(\mathrm{e}^{\mathrm{j} \omega}) \widehat{\varphi}(\omega).  \label{eq:Bphi}
	\end{align}
	Set $\omega_0 \in (0,2\pi)$ and $\omega_1 = \omega_0 + 2 \pi$. The relation \eqref{eq:Bphi} imposes that $A(\mathrm{e}^{\mathrm{j} \omega_i}) $, $B(\mathrm{e}^{\mathrm{j} \omega_i}) $, and $\widehat{\varphi}(\omega_i)$ are non-zero for $i=1,2$, and 
	\begin{equation}
	 \frac{1- \mathrm{e}^{-\mathrm{j}\omega_i}}{\mathrm{j}\omega_i} A(\mathrm{e}^{\mathrm{j} \omega_i}) \widehat{\varphi}(\omega_i) = B(\mathrm{e}^{\mathrm{j} \omega_i} ) \widehat{\varphi}(\omega_i).
	\end{equation}
	After simplifications, we deduce that 
	\begin{equation} \label{eq:strangeequality}
	\mathrm{j} \omega_i = \frac{(1- \mathrm{e}^{-\mathrm{j}\omega_i}) A(\mathrm{e}^{\mathrm{j} \omega_i}) }{B(\mathrm{e}^{\mathrm{j} \omega_i}) }.
	\end{equation}
	The right term in \eqref{eq:strangeequality} is equal for $\omega_0$ and $\omega_1$ by periodicity, while the left term is different. This contradicts our initial assumption, and implies Proposition \ref{prop:lesdeuxfontlapaires}. 
	\end{proof}

   \section{Approximation Properties of Hermite Splines} \label{sec:approximation}
Existing approaches have been proposed to characterize the behaviour of the approximation error for single~\cite{Blu1999a} and multi-generators~\cite{Blu1999b}. They however assume technical conditions that Hermite interpolation does not satisfy. 
We therefore formulate new theoretical tools to quantify the asymptotic constant of the approximation error of this scheme. Since functions are estimated with derivative samples in the Hermite case, we also propose to investigate the approximation error on the first derivative.
In our setting, other existing approximation schemes can be considered as well, allowing us to put the excellent approximation properties of Hermite splines in perspective with other related schemes, such as cubic B-splines and interlaced derivative sampling. 

\subsection{Generalized Sampling and Approximation Error}
The approximation of a  continuously-defined signal from a collection of its samples in a generalized sampling scheme relies on two ingredients: some basis functions $\bm{\varphi} = (\varphi_1, \ldots , \varphi_{N}) \in \left(L_2(\R)\right)^{N}$, and some \emph{sampling functions} $\bm{\tupvarphi} = (\tvarphi_1, \ldots , \tvarphi_{N})$ that are rapidly decaying generalized functions, including rapidly decaying $L_2$-integrable functions\footnote{We say that an $L_2$ function is \emph{rapidly decaying} if it decays faster than any polynomial at infinity. A generalized function in $\mathcal{S}'(\R)$ is rapidly decaying if its convolution with any infinitely differentiable and rapidly decaying function is a rapidly decaying function \cite{Schwartz1966distributions}. In particular, a rapidly decaying (generalized) function has an infinitely differentiable Fourier transform, which we shall rely on thereafter.} together with the Dirac impulse $\delta$, its derivative, and their shifts.
The term \emph{generalized} is motivated by the fact that sampling functions allow accessing more than the values of the signal at the sampling points~\cite{Papoulis1977}. The set of pairs of basis and sampling functions fully characterize an approximation scheme. Moreover, because we aim at a fair comparison, the quantity of information per unit of time should be equal among the considered schemes. This requires to slightly adapt the definition of $V(\bm{\varphi})$ given in~\eqref{eq:outspace}. From now, we shall consider
		\begin{align}\label{eq:outspace2}
		W(\bm{\varphi})&=\left\{ \sum_{i=1}^{N}\sum_{k\in\mathbb{Z}} c_i[k]\varphi_i( \cdot- Nk) : \mathbf{c} = (c_1, \ldots , c_N) \in (\ell_2(\mathbb{Z}))^{N} \right\},
		\end{align}
so that, for any $N \geq 1$, there is on average a single degree of freedom on each interval of size one.
The sampling and reconstruction problem is then formally defined as follows. The function $f$
is reconstructed by its approximation $\tilde{f}$ associated to the basis functions $\bm{\varphi}$ and for the sampling fuctions $\tilde{\bm{\varphi}}$, defined as
\begin{align}\label{eq:Qapprox}
\tilde{f}&= \sum_{i=1}^{N}\sum_{k\in\mathbb{Z}} \langle f, \tvarphi_i(\cdot-Nk) \rangle \varphi_i(\cdot-Nk) \in W(\bm{\varphi}).
\end{align}
We denote by $\mathcal{Q}_{\bm{\varphi}}^{\bm{\tilde{\varphi}}}$ the operator such that $\mathcal{Q}_{\bm{\varphi}}^{\bm{\tilde{\varphi}}} f = \tilde{f}$.
Hermite spline approximation thus corresponds to $N=2$ with basis functions $\varphi_1(t)=\phi_1(\frac{t}{2})$ and $\varphi_2(t)=2\phi_2(\frac{t}{2})$, and sampling functions $\tilde{\varphi}_1(t) = \delta(t)$ and $\tilde{\varphi}_2(t) = - \delta'(t)$.

The best approximation of a given scheme is obtained when the pairs of sampling and basis functions are properly chosen such that $\tilde{f}$ is the orthogonal projection of $f$ onto $W(\bm{\varphi})$. This implies imposing a particular condition~\cite{Blu1999b} on the sampling functions $\bm{\tupvarphi}$, namely that
\begin{align}\label{eq:phid}
	\bm{\tupvarphi} = \begin{pmatrix} \varphi_{1,d} \\ \vdots \\ \varphi_{N,d} \end{pmatrix}= \mathcal{F}^{-1}\{\mathbf{G}_{\bm{\varphi}}(\cdot)^{-1} \widehat{\bm{\varphi}}\},
\end{align}
where $\mathbf{G}_{\bm{\varphi}}$ is the Gram matrix of size $N\times N$ associated to $\bm{\varphi}$, given for $\omega \in \R$ by
\begin{align}\label{eq:gram}
	\mathbf{G}_{\bm{\varphi}} (\omega) 
	&=
	\sum_{k\in\Z} \widehat{\bm\varphi} (\omega +2k\pi)  \widehat{\bm\varphi}^{*T} (\omega +2k\pi) .
\end{align} 
This particular collection of sampling functions are called the \emph{dual functions} associated to $\bm{\varphi}$ and are denoted by $\bm{\varphi}_d$. 
Note that $\mathbf{G}_{\bm{\varphi}}(\omega)$ is invertible for every $\omega \in \R$ --- and therefore \eqref{eq:gram} is meaningful --- because \eqref{eq:Rieszcondition} is equivalent to $0 < A \leq \lambda_{\min}(\omega) \leq \lambda_{\max}(\omega)  \leq B < \infty$, where $\lambda_{\min}(\omega)$ ($\lambda_{\max}(\omega) $, respectively)  is the minimum (maximum, respectively) eigenvalue of $\mathbf{G}_{\bm{\varphi}}(\omega)$ \cite{Aldroubi1996oblique}.
When $\bm{\tupvarphi}$ is defined as~\eqref{eq:phid}, the operator $\mathcal{Q}_{\bm{\varphi}}^{\tilde{\bm{\varphi}}}$ is the orthogonal projector over $W(\bm{\varphi})$, and is denoted as $\mathcal{P}_{\bm{\varphi}} = \mathcal{Q}_{\bm{\varphi}}^{\bm{\varphi}_d}$. In this situation,~\eqref{eq:Qapprox} is reformulated as
	\begin{align} \label{eq:P}
	\tilde{f}(t)=\mathcal{P}_{\bm{\varphi}} f = \sum_{i=1}^{N} \sum_{k\in \Z} \langle f, \varphi_{i,d} \left( \cdot- Nk\right) \rangle\varphi_i\left(\cdot-Nk\right).
	\end{align} 

To simplify the notation, we shall write $\mathcal{P}= \mathcal{P}_{\bm{\varphi}}$ and $\mathcal{Q} = \mathcal{Q}_{\bm{\varphi}}^{\tilde{\bm{\varphi}}}$ thereafter. 
The quality of the approximation is then evaluated in terms of approximation error which is expressed as
\begin{align}
E_{\bm\varphi}^{\bm\tupvarphi}(f) =  \lVert f - \tilde{f}\rVert_{L_2} = \lVert f - \mathcal{Q} f \rVert_{L_2}.
\end{align}
When $\tilde{f}=\mathcal{P} f$, the error is denoted as $E_{\bm\varphi}(f)$. A direct implication is that $E_{\bm{\varphi}}(f) = E_{\bm{\varphi}}^{\bm{\varphi}_d} (f) \leq E_{\bm\varphi}^{\bm\tupvarphi}(f)$, reaching the equality if and only if $\bm\tupvarphi=\bm{\varphi}_d$.

Up to now, we considered approximation schemes with (generalized) samples taken on the integer grid, corresponding to a sampling step size $T=1$. This parameter affects the coarseness of the approximation: when $T\rightarrow 0$, the error is expected to vanish. For $T>0$, the approximation space~\eqref{eq:outspace2} becomes
	\begin{align}\label{eq:VT}
		W_T(\bm{\varphi}) = \left\{ \sum_{i=1}^{N}\sum_{k\in \Z} c_i[k] \varphi_i \left( \frac{\cdot}{T} - Nk \right) \ : \ \mathbf{c} \in (\ell_2(\Z))^{N} \right\}
	\end{align}
and the approximation of $f$ is given by
	\begin{align} \label{eq:QT}
	\tilde{f}_T=\mathcal{Q}_T f = \sum_{i=1}^{N} \sum_{k\in \Z} \left\langle f, \frac{1}{T}\tilde{\varphi}_i \left( \frac{\cdot}{T} - Nk\right) \right\rangle \varphi_i\left(\frac{\cdot}{T}-Nk\right),
	\end{align}
with resulting error
\begin{align}
	E_{\bm\varphi}^{\bm\tupvarphi}(f,T) = \lVert f - \mathcal{Q}_T f \rVert_{L_2}.
	\end{align}
The orthogonal projector~\eqref{eq:P} and its associated error are easily reformulated accordingly. 

A number of hypotheses on the basis functions $\bm{\varphi}$ and the sampling functions $\bm{\tupvarphi}$ have to be met in order to study the errors $E_{\bm{\varphi}}(f,T)$ and $E_{\bm{\varphi}}^{\bm{\tupvarphi}}(f,T)$ in terms of rate of decay and asymptotic constant. The first one is the Riesz basis condition~\eqref{eq:Rieszcondition}, which ensures a unique and stable representation. The second one is the basis functions order. 
\begin{definition} \label{def:orderL}
A family of $N\leq 1$ basis functions $\bm{\varphi} = (\varphi_1,\hdots,\varphi_N)$ is \emph{of order $L$} if they can reproduce polynomial functions up to degree $(L-1)$, meaning that, for $\ell =0,\ldots,(L-1)$, there exists sequences $c_{i,\ell}$ such that
	\begin{align}\label{eq:reprodL}
	t^\ell=\sum_{i=1}^N\sum_{k\in\mathbb{Z}} c_{i,\ell}[k]\varphi_i(t-Nk).
	\end{align}
\end{definition}
This condition is known to be equivalent to the Strang and Fix conditions~\cite{Strang1971}. When it  is met, the decrease of the optimal error $E_{\bm\varphi}(f,T)$ is bounded from above by $T^L$\cite{Blu1999b}. The last important condition is that the sampling and basis functions are quasi-biorthonormal of order $L$.
\begin{definition}  \label{def:bioL}
Two families of basis $\bm\varphi$ and sampling functions $\bm\tupvarphi$ are \emph{quasi-biorthonormal of order $L$} if the basis functions are of order $L$ and, for the dual function $\bm\varphi_d$ given by~\eqref{eq:phid} and all $ \ell=0, \ldots, (L-1)$, we have
		 \begin{align}\label{eq:momentseq}
		\int_{\R} t^\ell \bm\tupvarphi(t)\dint t = \int_{\R} t^\ell \bm\varphi_d(t)\dint t .
	\end{align}
\end{definition}
It is worth noting that \eqref{eq:momentseq} is a slight abuse of notation, since the $\tilde{\varphi}_i$ are not necessarily defined pointwise. However, they are assumed to be rapidly decaying generalized functions and can therefore be taken against a slowly growing smooth functions such as $t\mapsto t^\ell$.
The quasi-biorthonormality ensures that the decrease of the error $E_{\bm\varphi}^{\bm\tupvarphi}(f,T) $ is also bounded by $T^L$\cite{Blu1999b}.
Finally, the rate of decay of the approximation error being under control in all generality, the asymptotic constant can be obtained as 
\begin{align}\label{eq:constantdep}
C_{\bm{\varphi}}^{\bm{\tupvarphi}} (f) = \lim_{T \rightarrow 0} T^{-L} E_{\bm{\varphi}}^{\bm{\tupvarphi}} (f, T),
\end{align}
which, in practice can be computed relying on a Fourier-domain approximation error kernel. For more details, we refer the interested reader to~\cite{Unser1998,Unser1997,Uhlmann2017Thesis} and references therein.

\subsection{Approximation Constants of Irregular Sampling Schemes}
Our goal in this section is to extend the range of applicability of the main results of \cite{Blu1999b} so as to include Hermite spline approximation. 
The original framework is indeed restricted to sampling functions $\tilde{\varphi}$ with bounded Fourier transforms. While this allows considering the case of interpolation, it excludes Hermite spline approximation since the Fourier transform $\mathrm{j}\omega$ of $\delta'$ is unbounded. 

Our development follows the key contributions of \cite{Blu1999b}. We therefore only detail the adaptations that are required in our case. 
We start with a brief summary of the general approach, which is common to many works of approximation theory in shift-invariant spaces.
We first introduce the kernels associated to the functions $\bm{\varphi}, \bm{\tilde{\varphi}}$ as
 \begin{align}
	\mathcal{E}_{\min}(\omega) =&  1 + \widehat{\bm{\varphi}}^{*T}(\omega) \mathbf{G}_{\bm\varphi}^{-1}(\omega)  \widehat{\bm{\varphi}}(\omega),\label{eq:eopt} \\
	\mathcal{E}_{\mathrm{res}} (\omega) =& (\widehat{\bm{\tupvarphi}}-\widehat{\bm{\varphi}_d})^{*T}(\omega) \mathbf{G}_{\bm{\varphi}}(\omega)(\widehat{\bm{\tupvarphi}}-\widehat{\bm{\varphi}_d})(\omega) , \\
	\mathcal{E} (\omega) =& \mathcal{E}_{\min}(\omega) + \mathcal{E}_{\mathrm{res}} (\omega),\label{eq:estd}
\end{align}
where $\mathbf{G}_{\bm{\varphi}}$ is the Gram matrix~\eqref{eq:gram}. The kernel $\mathcal{E}_{\min}$ relates to the \emph{minimum error} case achieved using the orthogonal projector (\textit{i.e.}, $\bm\tupvarphi=\bm\varphi_d$), and $\mathcal{E}_{\mathrm{res}}$ to the \emph{residual error} arising when using sampling functions that differ from the dual functions. We furthermore note that $\mathcal{E}_{\mathrm{res}}(\omega)=0$ when $\bm\tupvarphi=\bm\varphi_d$, as expected.
The key ideas of relying on $\mathcal{E}$ are as follows.
 \begin{itemize}
 	\item The kernel $\mathcal{E}$ measures the approximation power of a reasonable approximation scheme $(\bm{\varphi}, \bm{\tilde{\varphi}})$ in the sense that $\lVert f -  \mathcal{Q}_T f \rVert_{L_2} \approx \left( \int_{\R} \lvert \widehat{f} (\omega) \rvert^2 \mathcal{E}( T\omega) \mathrm{d}\omega \right)^{1/2}$ for small $T>0$.
 	\item The precise behaviour of $\lVert f  - \mathcal{Q}_T f \rVert_{L_2}$ is then deduced from that of $\mathcal{E}$ around the origin. It depends on the approximation order $L$ of the scheme and, typically, behaves like $C T^{L}$. The constant $C$ depends on the function to approximate $f$ and on the scheme $(\bm{\varphi},\bm{\tilde{\varphi}})$ via the Taylor expansion of $\mathcal{E}$. 
 	\end{itemize}
We give a precise meaning to these two points in Propositions \ref{prop:makingthefirstbound} and \ref{prop:taylorarrive} below.
Before that, we recall an important lemma taken from \cite{Blu1999b} that will play a fundamental role in our proofs.

	\begin{lemma} \label{lemma:blunser}
	For $k \geq 0$, we set $\widehat{f}_k (\omega) = \widehat{f}(\omega) 1_{k/T \leq \lvert \omega \rvert < (k+1) / T}$.
 Then,  the following relations hold:
 	\begin{align}
 	&\widehat{f} = \sum_{k \geq 0} \widehat{f}_k;  \label{eq:truciobvious}\\
	&\lVert f_k - \mathcal{Q}_T f_k \rVert_{L_2}^2 = \int_{\R} \lvert f_k(\omega) \rvert^2 \mathcal{E}(T\omega) \mathrm{d}\omega \text{ for } k \geq 0; \text{ and} \label{eq:controlfk}\\
 	& \big\lvert \lVert f - \mathcal{Q}_T f \rVert_{L_2} - \lVert f_0 - \mathcal{Q}_T f_0 \rVert_{L_2} \big\rvert \leq \sum_{k > 0 } \lVert f_k - \mathcal{Q}_T f _k \rVert_{L_2}.  \label{eq:1forprop3}
 	\end{align}
	\end{lemma}
	
	The equality \eqref{eq:truciobvious} is obvious. 
	The two next relations comes from \cite[Theorem 1]{Blu1999b} and we have simply reformulated the results with our notation. First, it is proven (see (27) in \cite{Blu1999b}) that $\lVert f - \mathcal{Q}_T f\rVert_{L_2} = 	\int_{\R} \lvert f_k(\omega) \rvert^2 \mathcal{E}(T\omega) \mathrm{d}\omega$ as soon as $\widehat{f}(\omega) \widehat{f}(\omega - n/T) = 0$ for any $\omega\in \R$ and $n\in \Z$, a condition that is satisfied by the $f_k$ by construction, giving \eqref{eq:controlfk}.
	Moreover, the inequality \eqref{eq:1forprop3} appears in the proof of \cite[Theorem 1]{Blu1999b} (see (63)).
	
 	\begin{proposition} \label{prop:makingthefirstbound}
 		Let $(\bm{\varphi}, \bm{\tilde{\varphi}})$ be a set of $N$ basis and sampling functions that are biorthonormal and provide an approximation scheme of order $L$. We assume moreover that the kernel $\mathcal{E}$ given by \eqref{eq:estd} satisfies  
 		\begin{align}\label{eq:boundEhyp}
 		\lvert \mathcal{E}(\bm{\omega}) \rvert \leq C^2 \max(1 , \lvert \omega \rvert^{2p})
 		\end{align}
 		 for some $C>0$, some integer $0 \leq p \leq L$, and every $\omega \in \R$. Then, for every $f \in W^{L+1}_2(\R)$, we have
 		\begin{align} \label{eq:firstboundQtf}
 		\lVert f - \mathcal{Q}_T f \rVert_{L_2} = \left( \int_{\R} \lvert \widehat{f} (\omega) \rvert^2 \mathcal{E}(T\omega) \mathrm{d}\omega \right)^{1/2} + O(T^{L+1}).
 		\end{align} 
 	\end{proposition}
 	 	
 	 The case $p=0$ corresponds to $\mathcal{E}$ being bounded and can be found in \cite[Theorem 1]{Blu1999b}.
 	\begin{proof}
 	 We set $\widehat{f}_k (\omega) = \widehat{f}(\omega) 1_{k/T \leq \lvert \omega \rvert < (k+1) / T}$. Then, we have 
	\begin{align} \label{eq:inequalityagainandagain}
		\Big\lvert \lVert f  & - \mathcal{Q}_T f \rVert_{L_2}  - \left( \int_{\R} \lvert \widehat{f} (\omega) \rvert^2 \mathcal{E}(T\omega) \mathrm{d}\omega \right)^{1/2}  \Big\rvert  \nonumber \\
	& \leq  
				\left\lvert \lVert f - \mathcal{Q}_T f \rVert_{L_2} - \lVert f_0 - \mathcal{Q}_T f_0 \rVert_{L_2} \right\rvert 	+
			\left\lvert \lVert f_0 - \mathcal{Q}_T f_0 \rVert_{L_2} - \left( \int_{\R} \lvert \widehat{f} (\omega) \rvert^2 \mathcal{E}(T\omega) \mathrm{d}\omega \right)^{1/2}  \right\rvert
			\nonumber \\
		&	\overset{(i)}{\leq} 
			\sum_{k > 0 } \lVert f_k - \mathcal{Q}_T f _k \rVert_{L_2} +
			\left\lvert \lVert f_0 - \mathcal{Q}_T f_0 \rVert_{L_2} - \left( \int_{\R} \lvert \widehat{f} (\omega) \rvert^2 \mathcal{E}(T\omega) \mathrm{d}\omega \right)^{1/2}  \right\rvert \nonumber \\
		& = 	(I) + (II),	
	\end{align}	 	
	where we used \eqref{eq:1forprop3} in $(i)$. 
	As a consequence, \eqref{eq:firstboundQtf} follows if one shows that the two terms $(I)$ and $(II)$  in \eqref{eq:inequalityagainandagain} are $O(T^{L+1})$.
	Using \eqref{eq:controlfk}, we have
	\begin{align}
		\lVert f_k - \mathcal{Q}_T f_k \rVert_{L_2}^2 
		&=
		\int_{k/T < \lvert \omega \rvert \leq (k+1)/T} \lvert f_k(\omega) \rvert^2 \mathcal{E}(T\omega) \mathrm{d}\omega \nonumber \\
&= 
\int_{k/T < \lvert \omega \rvert \leq (k+1)/T} \lvert f_k(\omega) \rvert^2 \lvert \omega \rvert^{2(L+1)} \frac{T^{2p}}{\lvert \omega \rvert^{2(L-p+1)}} \frac{ \mathcal{E}(T\omega)}{ T^{2p} \lvert \omega \rvert^{2p}}  \mathrm{d}\omega \nonumber \\	
 		& \leq
 		 C^2 \frac{T^{2p}}{(k/T)^{2(L-p+1)}} \int_{\R}  \lvert f_k(\omega) \rvert^2 \lvert \omega \rvert^{2(L+1)} \mathrm{d}\omega \nonumber \\	
 		&= C^2 \frac{T^{2(L+1)}}{k^{2(L-p+1)}} \lVert f_k^{(L+1)} \rVert_{L_2}^2,
\end{align}
where the inequality is due to $\lvert \omega \rvert \geq k/T$ over the domain and to \eqref{eq:boundEhyp}, which implies that $C^2 \geq \sup_{\omega \in \R} \mathcal{E}(\omega) / \max(1, \lvert \omega\rvert^p) \geq \sup_{\lvert \omega \rvert \geq 1} \mathcal{E}(\omega) / \lvert \omega\rvert^p$. Summing over $k>0$, we deduce that
\begin{align} \label{eq:boundI}
	(I)
	& \leq C T^{L+1} \sum_{k>0} \frac{1}{k^{L-p+1}} \lVert f_k^{(L+1)}\rVert_{L_2} \nonumber \\
	& \leq C T^{L+1} \left( \sum_{k>0} \frac{1}{k^{2(L-p+1)}} \right)^{1/2}
	\left( \sum_{k>0} \lVert f_k^{(L+1)}\rVert_{L_2}^2\right)^{1/2} \nonumber \\
	&= C T^{L+1} \sqrt{\zeta(2(L-p+1))} \lVert (f-f_0)^{(L+1)} \rVert_{L_2} = O(T^{L+1}),
\end{align}
where the second inequality is derived from Cauchy-Schwarz, and $\zeta$ is the Riemann zeta function.
For the second term, we remark that, again due to \cite[$(27)$]{Blu1999b}, $\lVert f_0 - \mathcal{Q}_T f_0 \rVert_{L_2}^2   = \int_{\lvert \omega \rvert \leq 1/T} \lvert \widehat{f}(\omega)\rvert^2 \mathcal{E}(T\omega) \mathrm{d}\omega$. Therefore, using the relation $\lvert \lVert g  \rVert_{L_2} - \rVert h \rVert_{L_2} \rVert \leq \lVert g - h \rVert_{L_2}$, which is a consequence of the Minkowski inequality, we deduce that
\begin{align} \label{eq:boundII}
(II)
 & \leq \left( \int_{\lvert \omega \rvert > 1/T } \lvert \widehat{f} (\omega) \rvert^2 \mathcal{E}(T\omega) \mathrm{d}\omega\right)^{1/2} \nonumber \\
 &= 
 \left( \int_{\lvert \omega \rvert > 1/T } \lvert \widehat{f} (\omega) \rvert^2 \lvert \omega \rvert^{2(L+1)} \frac{T^{2p}}{\lvert \omega \rvert^{2(L-p+1)}} \frac{\mathcal{E}(T\omega)}{T^{2p}\lvert \omega \rvert^{2p}} \mathrm{d}\omega\right)^{1/2} \nonumber \\
& \leq 
C T^{p} T^{L-p+1} \left( \int_{\lvert \omega \rvert > 1/T } \lvert \widehat{f} (\omega) \rvert^2 \lvert \omega \rvert^{2(L+1)}\mathrm{d}\omega \right)^{1/2} \nonumber \\
&= C T^{L+1} \lVert (f-f_0)^{(L+1)} \rVert_{L_2} = O(T^{L+1}),
\end{align}
where the inequality once more follows from $\lvert \omega \rvert \geq 1/T$ over the domain and from our assumption \eqref{eq:boundEhyp}.  Combining \eqref{eq:boundI} and \eqref{eq:boundII} in \eqref{eq:inequalityagainandagain} completes the proof.
 	\end{proof}
 	
 	Our next result connects the expression $\int_{\R} \lvert \widehat{f}(\omega)\rvert^2 \mathcal{E}(Tw) \mathrm{d} \omega$ to the Taylor expansion of $\mathcal{E}$.
 	
 	\begin{proposition} \label{prop:taylorarrive}
Let $(\bm{\varphi}, \bm{\tilde{\varphi}})$ be a set of $N$ basis and sampling functions that are biorthonormal and provide an approximation scheme of order $L$. We assume moreover that the kernel $\mathcal{E}$ given by \eqref{eq:estd} is $(2L+1)$-times continuously differentiable and satisfies 
\begin{align} \label{eq:boundEderiv}
\lvert \mathcal{E}^{(2L+1)}({\omega}) \rvert \leq C^2 \max(1 , \lvert \omega \rvert^{2p})
\end{align}
 for some $C>0$, some integer $0 \leq p \leq L$, and every $\omega \in \R$. Then, for every $f \in W^{L+p+1/2}_2(\R)$, we have
\begin{align}\label{eq:controlerror2}
	\int_{\R} \lvert \widehat{f}(\omega)\rvert^2 \mathcal{E}(Tw) \mathrm{d} \omega
	=  {\frac{\mathcal{E}^{(2L)}(0)}{(2L)!}} \lVert f^{(L)} \rVert_{L_2}^2 T^{2L} + O(T^{2L+1}).
	\end{align}
 	\end{proposition}

\begin{proof}
	The kernel $\mathcal{E}$ being symmetric, we deduce that $\mathcal{E}^{(2k+1)}(0) = 0$ for $k=0,\ldots, (L-1)$. 
	Moreover, the condition of $(\bm{\varphi}, \bm{\tilde{\varphi}})$ ensures that $\mathcal{E}^{(2k)}(0)=0$ for $k=0,\ldots, (L-1)$, together with $\mathcal{E}^{(2L)}(0) \neq 0$. Therefore, $\mathcal{E}$ being $(2L+1)$-times continuously differentiable, the Taylor expansion around the origin is given by
	\begin{align} \label{eq:taylorstuff}
		\mathcal{E}(\omega) = \frac{\mathcal{E}^{(2L)}(0)}{(2L)!} \omega^{2L} + \frac{\mathcal{E}^{(2L+1)}(\omega \theta)}{(2L+1)!} w^{2L+1},
	\end{align}
	with $\theta = \theta(\omega) \in (0,1)$. Using \eqref{eq:taylorstuff}, we deduce that
	\begin{align} \label{eq:controlemoica}
	\int_{\R} \lvert \widehat{f}(\omega)\rvert^2 \mathcal{E}(Tw) \mathrm{d} \omega
	&=
	\int_{\R} \lvert \widehat{f}(\omega)\rvert^2 \frac{\mathcal{E}^{(2L)}(0)}{(2L)!} T^{2L} \omega^{2L} \mathrm{d} \omega \nonumber \\
 &  \quad 	+ 
	\frac{T^{2L+1}}{(2L+1)!} \int_{\R} \lvert \widehat{f} (\omega)\rvert^2 \omega^{2L+1} \mathcal{E}^{(2L+1)}(\omega T \theta(\omega T) ) \mathrm{d}\omega
	\nonumber \\
 	&= 
 	\frac{\mathcal{E}^{(2L)}(0)}{(2L)!} \lVert f^{(L)} \rVert_{L_2}^2 T^{2L} \nonumber \\
 	& \quad +	
	\frac{T^{2L+1}}{(2L+1)!} \int_{\R} \lvert \widehat{f} (\omega)\rvert^2 \omega^{2L+1} \mathcal{E}^{(2L+1)}(\omega T \theta(\omega T) ) \mathrm{d}\omega.
	\end{align}
	Due to \eqref{eq:boundEderiv} and $0 < \theta(\omega T) < 1$,
	\begin{align}
		\mathcal{E}^{(2L+1)}(\omega T \theta(\omega T)) \leq C^2 \max(1,\lvert \omega \rvert^{2p} T^{2p} ) \leq C^2 (1 + \lvert \omega \rvert^{2p} T^{2p}), 
	\end{align}
	from which we deduce that
	$ \int_{\R} \lvert \widehat{f} (\omega)\rvert^2 \omega^{2L+1} \mathcal{E}^{(2L+1)}(\omega T \theta(\omega T) ) \mathrm{d}\omega \leq C^2 (\lVert f^{(L+1/2)} \rVert_{L_2}^2 + T^{2p} \lVert f^{L+p+1/2} \rVert_{L2}^2)$.  Injecting this to \eqref{eq:controlemoica} implies
	\begin{align}
	 \int_{\R} \lvert \widehat{f}(\omega)\rvert^2 \mathcal{E}(Tw) \mathrm{d} \omega
	-
	\frac{\mathcal{E}^{(2L)}(0)}{(2L)!} \lVert f^{(L)} \rVert_{L_2}^2 T^{2L} 
		  &=  C^2 \frac{T^{2L+1}}{(2L+1)!}  (\lVert f^{(L+1/2)} \rVert_{L_2}^2 \nonumber \\
		  & \hphantom{=  C^2 \frac{T^{2L+1}}{(2L+1)!}  (\lVert}+ T^{2p} \lVert f^{(L+p+1/2)} \rVert_{L2}^2) \nonumber \\
	  & = O(T^{2L+1}), 
	  	\end{align}
 which concludes the proof.
	\end{proof}
	
	Note that we use derivatives of fractional order $\gamma = L+1/2$ and $\gamma = L+p+1/2$ in the proof, which are defined in the Fourier domain as $\mathcal{F} \{f^{(\gamma)} \} (\omega) = (\mathrm{j} \omega)^{\gamma} \widehat{f}(\omega)$.
		Finally, we conclude with an extension of \cite[Theorem 4]{Blu1999b} to sampling functions that are not necessarily bounded in Fourier domain. 
	
	\begin{theorem} \label{theo:extendblunser}
	We consider an approximation scheme $(\bm{\varphi},\bm{\tilde{\varphi}})$ with $N$ basis and sampling functions such that
	\begin{itemize}
	\item the basis functions are rapidly decaying $L_2$ functions such that the family $\{ \varphi_{i} (\cdot - Nk \}_{i = 1\cdots N, k \in \Z}$ is a Riesz basis in the sense of \eqref{eq:Rieszcondition}, with approximation power of order $L\geq 1$ (see Definition \ref{def:orderL});
	\item The sampling functions are rapidly decaying generalized functions such that $(\bm{\varphi},\bm{\tilde{\varphi}})$ is biorthonormal of order $L$ (see Definition \ref{def:bioL});
	\item There exists an integer $0\leq p \leq L$ such that, for any $\omega \in \R$, any $0\leq k \leq L$, and any $1 \leq i \leq N$,
		\begin{align}\label{eq:boundtildephi}
		\lvert \widehat{\tilde{\varphi}}_i^{(k)} (\omega) \rvert \leq C \max( 1 , \lvert \omega \rvert^p).
		\end{align}
		\end{itemize}
		Then, for any $f \in W^{L+\max(p+1/2 , 1)}(\R)$, we have that
		\begin{align}\label{eq:toutvabien}
		\lVert f - \mathcal{Q}_T f \rVert_{L_2} \underset{T\rightarrow 0}{\sim} \sqrt{\frac{\mathcal{E}^{(2L)}(0)}{(2L)!}} \lVert f^{(L)} \rVert_{L_2} T^L. 
		\end{align}
	\end{theorem}
	
	The two first conditions in Theorem \ref{theo:extendblunser} are necessary to ensure that the approximation scheme has an approximation power of order $L$. The last condition allow us to show that such an approximation power is attained with restricted condition on the sampling functions having possibly unbounded Fourier transform.

	\begin{proof}
			We prove that the conditions of Theorem \ref{theo:extendblunser} imply that we fulfill the hypotheses of both Propositions~\ref{prop:makingthefirstbound} and~\ref{prop:taylorarrive}. 
			
			Since ${\bm{\varphi}},\tilde{\bm{\varphi}}$ are rapidly decaying (generalized) functions, their Fourier transforms $\widehat{\bm{\varphi}}$ and $\widehat{\tilde{\bm{\varphi}}}$ are infinitely differentiable.
			The same holds true for $\widehat{\varphi}_d = \mathbf{G}_{\bm{\varphi}}^{-1} \widehat{\varphi}$ since $\mathbf{G}_{\bm{\varphi}}$ is a matrix-valued infinitely differentiable function with infinitely differentiable inverse $\mathbf{G}_{\bm{\varphi}}^{-1}$ due to the Riesz basis condition. This implies that $\mathcal{E}$ is infinitely differentiable and therefore $(2L+1)$-times continuously differentiable.
			
			The basis functions are rapidly decaying, hence in $L_1(\R)$, impling that $\widehat{\bm{\varphi}}$ is bounded.
			The Riesz basis condition then easily implies that both $\mathbf{G}_{\bm{\varphi}}$ and its inverse are bounded as functions of $\omega$. 
			It then follows that $\mathcal{E}_{\min}(\omega)$ is bounded, while $\mathcal{E}_{\mathrm{res}}(\omega)$ is dominated by $\lVert (\widehat{\tilde{\bm{\varphi}}} -   \widehat{\bm{\varphi}}_d ) (\omega) \rVert$, the latter being dominated by $\max(1, \lvert \omega\rvert^{2p})$ due to \eqref{eq:boundtildephi}. 
			It finally implies \eqref{eq:boundEhyp} for some constant $C>0$.
			Similarly, for $k\leq L$, the function $\widehat{\varphi}^{(k)}$ is bounded, being the Fourier transform of $t\mapsto t^k \varphi(t)$, which is in $L_1(\R)$ due to the rapid decay of $\varphi$. Again, this property is transferred to the derivative of $\mathbf{G}_{\bm{\varphi}}$ and its inverse. Then, exploiting Leibnitz rule, one shows that $\mathcal{E}_{\min}^{(2L+1)}$ is bounded, while $\mathcal{E}^{(2L+1)}$ is a sum of products constituted of bounded terms and of two terms of the form $\widehat{\tilde{\varphi}}_i(\omega)$, which are controlled by \eqref{eq:boundtildephi}. Putting things together, we easily deduce \eqref{eq:boundEderiv} for some constant $C>0$.
			
		Finally, the hypotheses of Propositions \ref{prop:makingthefirstbound} and \ref{prop:taylorarrive} are satisfied, implying \eqref{eq:firstboundQtf} and \eqref{eq:controlerror2}  together with
			\begin{align}
					\lVert f - \mathcal{Q}_T f \rVert_{L_2}= \sqrt{\frac{\mathcal{E}^{(2L)}(0)}{(2L)!}} \lVert f^{(L)} \rVert_{L_2} T^L + O(T^{L+1/2}).
			\end{align}
			This proves \eqref{eq:toutvabien}.
	\end{proof}
	
		Note that \cite[Theorem 4]{Blu1999b} corresponds to Theorem \ref{theo:extendblunser} with $p=1$, and with less restrictive assumptions on $\bm{\varphi}$ and $\bm{\tilde{\varphi}}$ (essentially, the rapid decay is replaced by polynomial decay adapted to $L$).
		The condition $f \in W^{L+ \max(p+1/2 , 1)}(\R)$ reflects the cost of covering the scenario of possibly irregular sampling functions with unbounded Fourier transform.

\subsection{Approximation Properties of Hermite Splines}
We  can now  evaluate the approximation error on $f$ in different frameworks, including the Hermite scheme.
In our analysis, for $(\bm{\varphi}, \bm{\tilde{\varphi}})$ being fixed, we also quantify the error on the derivative $f'$ when we approximate it by $(\mathcal{Q}_T f)'$.  We therefore study the quantities $\lVert f - \mathcal{Q}_T f \rVert_{L_2}$ and $ \lVert f' - (\mathcal{Q}_T f)' \rVert_{L_2}$.
	Knowing the order of approximation $L\geq 1$, the quality of the approximation is quantified by  the two asymptotic constants
		\begin{align}
		\mathbf{C}_{\bm{\varphi}}^{\bm{\tupvarphi}}
		= \left( \begin{array}{c}   {C}_{\bm\varphi,1}^{\bm\tupvarphi}  \\ {C}_{\bm\varphi,2}^{\bm\tupvarphi}  \end{array} \right)	
		= 
		\left( \begin{array}{c}
		\lim_{T \rightarrow 0} T^{-L} \lVert f^{(L)} \rVert_{L_2}^{-1} \lVert f - \mathcal{Q}_T f \rVert_{L_2} \\
		\lim_{T \rightarrow 0} T^{-(L-1)} \lVert f^{(L-1)} \rVert_{L_2}^{-1} \lVert f' - (\mathcal{Q}_T f)' \rVert_{L_2}
		\end{array} \right).
		\end{align}
The asymptotic constant of $C_{\bm{\varphi},2}^{\bm{\tupvarphi}}$ can be computed with the same tools as that of $C_{\bm{\varphi},1}^{\bm{\tupvarphi}}$. Using integration by parts, we indeed have that
\begin{align}\label{eq:QTp}
\left(\mathcal{Q}_T f\right)' &= \frac{1}{T} \sum_{i=1}^{N} \sum_{k\in \Z} \left\langle f, \frac{1}{T}\tilde{\varphi}_{i} \left( \frac{\cdot}{T} - k\right) \right\rangle \varphi'_i\left(\frac{\cdot}{T}-k\right) \\
&= \frac{1}{T} \sum_{i=1}^{N} \sum_{k\in \Z} \left\langle f', \frac{1}{T}\tilde{\varphi}_{i,\mathrm{int}} \left( \frac{\cdot}{T} - k\right) \right\rangle \varphi'_i\left(\frac{\cdot}{T}-k\right),
\end{align}
where the new sampling functions are best defined in the Fourier domain as
\begin{align} \label{eq:newtilde}
\widehat{\tilde{\varphi}_{i,\mathrm{int}}}(\omega)=-\frac{1}{\jj \omega} \widehat{\tilde{\varphi}_i}(\omega).
\end{align}
As one power of $T$ gets lost in the differentiation process, the rate of decay of the error on $f'$ is thus equal to $(L-1)$. Note that the apparent singularity in~\eqref{eq:newtilde} around $0$ is counter-balanced in the analysis since one only approximates functions $f'$ that are derivatives functions, hence for which $\widehat{f'} (\omega) = \mathrm{j} \omega \widehat{f} (\omega)$. 

In the Hermite framework, for which $N=2$, sampling functions are taken as $\tvarphi_1=\delta$ and $\tvarphi_2= - \delta'$ and basis functions as~\eqref{eq:phi1} and~\eqref{eq:phi2}. For comparison purpose, we also consider two relevant schemes that fit our analysis framework: classical cubic B-splines and interlaced derivative sampling. Cubic B-spline approximation corresponds to $N=1$, with $\tvarphi=\sum_{k\in\mathbb{Z}} (b^3)^{-1}[k]\delta(\cdot -k)$, where $(b_1^3)^{-1}$ is the direct B-spline filter, and $\tvarphi=\beta^3$ the cubic B-spline~\cite{Unser1999}. Interlaced derivative sampling can be defined in the more general framework of generalized sampling without band-limited constraints~\cite{Unser1998}. In this setting, $N=2$ and the sampling functions correspond to $\tvarphi_1=\delta$ and $\tvarphi_2= - \delta'\left(\cdot-\frac{1}{2}\right)$. 
The basis functions are constructed from the cubic B-spline to allow for a fair comparison, and are given by
\begin{align}
\widehat{\varphi}_1(\omega)&=\frac{3 \ee^{-2 \jj  \omega} \left(-1+ \ee^{\jj  \omega}\right)^4}{2 \omega^4},\label{eq:zerub1} \\
\widehat{\varphi}_2(\omega)&=\frac{\ee^{-2 \jj  \omega} \left(-1+\ee^{\jj  \omega}\right)^4 \left(1+\ee^{\jj  \omega} \left(-4+\ee^{\jj  \omega}\right)\right)}{\left(2-2 \ee^{2 \jj  \omega}\right) \omega^4}.\label{eq:zerub2}
\end{align}
in the Fourier domain. A comprehensive description of their derivation is provided in~\cite{Uhlmann2017Thesis}.
It is worth noting that the sampling functions do not correspond to the dual functions in these frameworks. This is easily motivated by practical considerations: the sampling process must, in practice, be implemented with digital filtering, excluding dual functions due to their continuous nature. The dual functions can, however, still be constructed from the Gram matrix following~\eqref{eq:phid} so as to estimate the optimal approximation error.

We now reveal the approximation power of these three different schemes. The results are known for the cubic B-splines~\cite{Blu1999a} and are included for comparison purposes. For Hermite approximation and interlaced sampling, they are deduced from Theorem \ref{theo:extendblunser}, and are not included in the multigenerator framework \cite{Blu1999b}, whose hypotheses exclude the use of derivative samples.

\begin{proposition}
Let $(\bm{\varphi}, \bm{\tvarphi})$ be one of the three approximation schemes considered above. Then, we have that
\begin{align} \label{eq:sixerreurs}
	\lVert f - \mathcal{Q}_T f \rVert_{L_2} & \underset{T\rightarrow 0}{\sim} \frac{1}{72 \sqrt{70}}
	\lVert f^{(4)} \rVert_{L_2} T^4
	, \text{ and} \nonumber \\
		\lVert f' - (\mathcal{Q}_T f)' \rVert_{L_2} & 
		\underset{T\rightarrow 0}{\sim} 
		\frac{1}{12 \sqrt{210}}
	\lVert f^{(3)} \rVert_{L_2} T^3,
\end{align}
for every $f \in W_2^{5}(\R)$ (cubic B-splines), or $f\in W_2^{11/2}(\R)$ (Hermite splines or interlaced derivative sampling).
Moreover, in the three cases, we have the following relation to the optimal approximation scheme associated to $\bm{\varphi}$: 
\begin{align}\label{eq:new}
\lVert f - \mathcal{Q}_T f \rVert_{L_2} & \underset{T\rightarrow 0}{\sim} \sqrt{\frac{10}{3}} \lVert f - \mathcal{P}_T f \rVert_{L_2}, \text{ and} \nonumber \\
\lVert f' - (\mathcal{Q}_T f)' \rVert_{L_2}  & \underset{T\rightarrow 0}{\sim}  \lVert f' - (\mathcal{P}_T f)' \rVert_{L_2},
\end{align}
for every $f \in W_2^{5}(\R)$  (cubic B-splines), or $f\in W_2^{11/2}(\R)$  (Hermite splines and interlaced derivative sampling).
\end{proposition}

\begin{proof}
	The three frameworks are readily known to define approximation schemes of order $L = 4$~\cite{Unser1999,Uhlmann2016,Unser1998}. All the considered basis functions specify a Riesz basis, are rapidly decaying (Hermite and cubic B-splines are compactly supported, while the basis functions for interlaced derivative sampling, despite being non-compactly supported, are exponentially decaying~\cite{Unser1998}), and reproduce polynomials up to degree $3$. The sampling functions are rapidly decaying generalized functions. Indeed, they are compactly supported for the Hermite and interlaced derivative sampling schemes. For cubic B-splines, we have seen that $\tvarphi=\sum_{k\in\mathbb{Z}} (b^3)^{-1}[k]\delta(\cdot -k)$, where the sequence $(b^3)^{-1}$ is known to be  exponentially decaying~\cite{Unser1999}, implying the result.
	In addition, for the Hermite and interlaced derivative sampling schemes (for cubic splines, respectively), \eqref{eq:boundtildephi} is clearly satisfied with $p=1$ ($p=0$, respectively).
	The conditions of Theorem \ref{theo:extendblunser} are therefore satisfied for $L=4$, implying \eqref{eq:toutvabien}. 
	
	The value of $\mathcal{E}^{(2L)}(0)$ is computed by computing the Taylor expansion of the kernel $\mathcal{E}$ around $0$\footnote{We relied on the technical computing software Mathematica $11$ for this task.}. 
	The analysis of the approximation error on the derivative follows the same principle, the kernel being given for $(\bm{\varphi}', \tilde{\bm{\varphi}}_{\mathrm{int}})$ according to \eqref{eq:newtilde}, giving \eqref{eq:sixerreurs}.
	
	We obtain the asymptotic behaviour of the optimal approximation errors $\lVert f - \mathcal{P}_T f \rVert_{L_2}$ and $\lVert f' - (\mathcal{P}_T f)' \rVert_{L_2}$ associated to $\bm{\tvarphi} = \bm{\varphi}_d$ in the same way, leading to \eqref{eq:new}.
\end{proof}

Our findings call for the following comments.
\begin{itemize}
	\item The three compared schemes have the same approximation order, the same approximation constant, and the same optimal approximation constant (associated to $(\bm{\varphi}, \bm{\varphi}_d)$) for reconstructing both $f$ and its derivative.
	
		\item In all cases, the sampling functions result in a near-to-optimal asymptotic constant $C_{\bm{\varphi},1}^{\bm{\tvarphi}}$ for the reconstruction of $f$. This minor discrepancy of a factor of $\sqrt{10/3} \approx 1.83$ is expected from the fact that the sampling functions are not dual functions. It is also remarkable that the reconstruction of the derivative, while not being associated to the dual functions, is associated to an optimal approximation constant $C_{\bm{\varphi},2}^{\bm{\tvarphi}}$. 
		
	\item 
		Dual functions, although offering the smallest approximation error, have strong practical disadvantages.
	First, the $\bm{\varphi}_d$, are non-compactly supported splines. 
	More importantly, they cannot be easily implemented as they do not have a digital filter equivalent (\emph{i.e.}, computing the $\langle f , \bm{\varphi}_d\rangle$ does not only depend on knowing $f$ and its derivative on a fixed grid, in contrast to usual cubic B-splines and Hermite splines). Finally, they do not possess a closed-form expression in general. For these reasons, the sampling functions $\bm{\tvarphi}$ classically used in the three considered approximation schemes, although non-optimal, are preferable in practice.
\end{itemize}

In Table~\ref{tab:errors}, we sum up these findings, including the asymptotic constants for cubic B-spline, interlaced derivative sampling and cubic Hermite splines and their comparison with optimal constants.
Hermite interpolation is thus not a unique way of approximating a function and its first derivative, even if one wishes the error to remain close to optimal. The notable difference lies in the fact that the Hermite scheme provides functions that are simultaneously of finite support, which is not the case for interlaced derivative sampling (see~\eqref{eq:zerub1} and~\eqref{eq:zerub2}), and interpolating, which is not the case for cubic B-splines. 

{\renewcommand{\arraystretch}{1.5}
	\begin{table}
	\footnotesize
\centering
\caption{Comparison of approximation methods.\label{tab:errors}}
\begin{tabular}{l|c|c|c}
  \hline
  Approximation Method & \multicolumn{1}{|p{1.4cm}|}{\centering Cubic \\ B-splines} & \multicolumn{1}{|p{2.8cm}|}{\centering Interlaced Derivative Sampling} & \multicolumn{1}{|p{1.3cm}}{\centering Hermite \\ Splines} \\
  \hline
    Digital filter implementation & \cmark & \cmark & \cmark \\ 
  Interpolating  & \xmark & \cmark & \cmark \\
  Finite Support & \cmark & \xmark & \cmark \\
  Closed-form expression & \cmark & \xmark & \cmark \\ \hdashline 
Rate of Decay ($L$) & $4$ & $4$ & $4$ \\
  Asymptotic Constant ($C_{\bm{\varphi},1}^{\bm{\tupvarphi}}$) & $\frac{1}{72\sqrt{70}}$ & $\frac{1}{72\sqrt{70}}$ & $\frac{1}{72\sqrt{70}}$ \\
  Ratio to optimal ($C_{\bm{\varphi},1}^{\bm{\tupvarphi}}$) & $\frac{\sqrt{3}}{\sqrt{10}}$ & $\frac{\sqrt{3}}{\sqrt{10}}$ & $\frac{\sqrt{3}}{\sqrt{10}}$ \\ \hdashline 
Rate of Decay ($L-1$) & $3$ & $3$ & $3$ \\
  Asymptotic Constant ($C_{\bm{\varphi},2}^{\bm{\tupvarphi}}$) & $\frac{1}{12\sqrt{210}}$ & $\frac{1}{12\sqrt{210}}$ & $\frac{1}{12\sqrt{210}}$ \\
  Ratio to optimal ($C_{\bm{\varphi},2}^{\bm{\tupvarphi}}$) & $1$ & $1$ & $1$ \\
  \hline
\end{tabular}
\end{table}


\section{Concluding Remarks} \label{sec:conclusion}
Our work focused on the formal investigation of two practical aspects of Hermite splines, namely their short support and their good approximation properties. 
We show that Hermite splines are of minimal support among pairs of functions with similar reproduction properties, and provide a framework to quantify their approximation power. The latter result not only allows us to prove that Hermite splines are asymptomatically identical to cubic B-splines, but also offers a more general framework for the quantitative approximation of functions and their derivatives.

In summary, Hermite splines are found to offer an approximation scheme that (1) has the same approximation power than the notorious cubic B-splines, (2) is interpolating (possibly with the derivative), (3) is based on maximally localized compactly supported basis functions. The resulting cost is the need to use two basis functions instead of a single one, and have access to derivative samples.

\section{Acknowledgements}
This work was supported by core funding from EMBL, the Swiss National Science Foundation under Grant $200020\_162343/1$, and the European Research Council under Grant H2020-ERC (ERC grant agreement No 692726 - GlobalBioIm).

\textit{Competing interests:} The authors have no competing interests.

\bibliographystyle{IEEEtran} 
\bibliography{biblio}
\end{document}